\documentclass[a4paper,12pt,onecolumn]{article}
\usepackage[pdftex]{graphicx} 
\usepackage{fancyhdr}
\usepackage{caption} 
\usepackage{subcaption}
\usepackage{natbib}
\usepackage{float}
\usepackage{amssymb}
\usepackage{amsthm}
\usepackage[intlimits]{amsmath} 
\usepackage{booktabs} 
\usepackage{geometry}
\usepackage{dsfont}
\usepackage{enumitem}
\usepackage{bbm}
\usepackage{bm,times}
\usepackage{attachfile}
\usepackage[table]{xcolor}
\usepackage{array}
\newtheorem{theorem}{Theorem}
\newtheorem{lemma}[theorem]{Lemma}

\newtheorem{corollary}[theorem]{Corollary}
\newtheorem{definition}[theorem]{Definition}

\newtheorem{conjecture}[theorem]{Conjecture}

\newcommand{\dd}{\,\mathrm{d}}

\newcommand{\E}{{\mathbb E}}

\newcommand{\R}{{\mathbb R}}

\newcommand{\CC}{{\mathcal{C}}}

\newcommand{\bsx}{\boldsymbol x}

\newcommand{\bsU}{\boldsymbol U}
\newcommand{\bsX}{\boldsymbol X}
\newcommand{\bsY}{\boldsymbol Y}
\newcommand{\bsZ}{\boldsymbol Z}
\newcommand{\bsnull}{\boldsymbol 0}
\newcommand{\bsone}{\boldsymbol 1}
\newcommand{\one}{\mathds{1}}

 \title{$L_p$-norm spherical copulas}

\author{ Carole Bernard \\
\texttt{carole.bernard@grenoble-em.com}\\
Department of Accounting, Law and Finance\\
Grenoble Ecole de Management,\\
Grenoble, France. 
\and 
Alfred M\"{u}ller\\
\texttt{mueller@mathematik.uni-siegen.de}\\
Department of Mathematics,\\
University of Siegen,\\
Siegen, Germany. 
\and
Marco Oesting\\
\texttt{marco.oesting@mathematik.uni-stuttgart.de}\\
Stuttgart Center for Simulation Science (SC SimTech)\\
and Institute for Stochastics and Applications, \\
University of Stuttgart,\\
Stuttgart, Germany.
}

\begin{document}
\maketitle

\begin{abstract}
In this paper we study $L_p$-norm spherical copulas for arbitrary $p \in [1,\infty]$ and arbitrary dimensions. The study is motivated by a conjecture that these distributions lead to a sharp bound for the value of a certain generalized mean difference. 
We fully characterize conditions for existence and uniqueness of  $L_p$-norm spherical copulas. Explicit formulas for their densities and correlation coefficients are derived and the distribution of the radial part is determined. Moreover, statistical inference and efficient simulation are considered. 
\end{abstract}



\section{Introduction}
Starting from multivariate normal distributions, spherical or more generally elliptical distributions have been widely studied in probability and statistics. In the absolutely continuous case, spherical distributions have the appealing property that the contour lines of their probability density functions 
are circles, i.e.\ the value $f(x)$ of the probability density function depends on the Euclidean norm $\|\bsx\|_2$ of its argument only. \citet{gupta-song-97} suggested to generalize this property by replacing the Euclidean norm $\|\bsx\|_2$ by an arbitrary $L_p$-norm  $\|\bsx \|_p := (\sum_{i=1}^n |x_i|^p)^{1/p}$  with $1 \leq p < \infty$ to define $L_p$-norm spherical distributions. In other words, a distribution with probability density function $f:\R^n \to [0,\infty)$ is called $L_p$-norm spherical if  $f(\bsx) = g(\|\bsx \|_p)$ for some function $g:[0,\infty) \to [0,\infty)$. Notice that having a density that only depends on an $L_p$-norm must not be confused with the case of having a multivariate survival function which depends only on the $L_p$-norm. These distributions have also been considered in the literature. For a recent study see e.g. \cite{mai-wang}. 

It is a natural question to ask for spherical distributions with some prescribed marginal distributions. In the classical $L_2$-norm case this problem is well studied \citep[cf.][for instance]{eaton1981,joe1997}.
Of particular interest are $L_2$-norm spherical distributions with uniform marginal distributions, i.e.\ $L_2$-norm spherical copulas, that are investigated
in \citet{schwarz1985} and \cite{perlman-wellner-11}.
Motivated by recent findings about the role of spherical symmetric copulas in some optimization problem in \cite{BM2020}, we will study in this paper $L_p$-norm spherical copulas for general $p \in [1,\infty]$.

More precisely, \cite{BM2020} investigate a version of  ``dependence uncertainty bounds'' in which they maximize the so-called multivariate Gini mean difference $S(C,C)$ over all possible bivariate copulas $C$. The multivariate Gini mean difference $S(C,C)$ is defined in \cite{koshevoy1997} as the expected distance between two independent $d$-dimensional random vectors $\bsX$ and $\bsY$ with distribution $C$, i.e. $S(C,C) = \E\left( \| \bsX - \bsY \|_2 \right)$, see also e.g., \cite{yitzhaki2003} for an overview on Gini indices and their applications.  This expected distance between independent copies of random vectors is the main ingredient in computing the energy distance between two probability distributions, and enters also  in the definition of the energy score that measures a distance between an observation and a distribution. There are various applications  of the energy distance and the energy score, in particular for goodness-of-fit tests in multivariate statistics \citep[see][for a review]{szekely2017}, and  in multivariate probabilistic forecasting, as the  energy score is  a strictly proper scoring rule for multivariate distributions \cite[see][]{gneiting2007}. It is an important question, how sensitive such a scoring rule is with respect to misspecifications of the dependence structure. This problem has been studied by \cite{pinson2013} and \cite{ziel2019}, which inspired \cite{BM2020} to investigate dependence uncertainty bounds for such quantities.

The energy score can be generalized by introducing a parameter $\beta \in (0,2)$ and replacing $S(C,C)$ by a generalized expected distance between two independent $d$-dimensional random vectors of a multivariate copula $C$  defined as 
$S_\beta(C,C) = \E\left( \| \bsX - \bsY \|^\beta_2 \right).$ It is still an open problem to determine dependence uncertainty bounds for this more general case. However, numerical experiments lead to the conjecture that the $L_p$-norm spherical copulas introduced in this paper may play a key role in determining dependence uncertainty upper bound for this generalized Gini mean difference. This motivated us to study this object in detail.

In Section \ref{S1}, we recall results by \citet{gupta-song-97} and provide the definitions of $L_p$-norm spherical distributions and copulas for $1 \le p < \infty$. In Section \ref{sec:bivariate} we focus on the interesting case of two dimensions and study $L_p$-norm spherical bivariate copulas. A generalization to multivariate copulas is then presented in Section \ref{sec:multivariate}. Finally, we present {$L_\infty$-norm spherical distributions and copulas in Section \ref{S4}.

\section{$L_p$-norm spherical distributions and copulas \label{S1}}

We build on the work of \citet{gupta-song-97} who consider $L_p$-norm spherical distributions on $\R^n$ for $1 \le p < \infty$. Assume that the density $f:\R^n \to [0,\infty)$ of a random vector $\bsX = (X_1,\ldots,X_n)$ with an absolutely continuous distribution depends only on the norm $\|\bsx \|_p := (\sum_{i=1}^n |x_i|^p)^{1/p}$, i.e.\ $f(\bsx) = g(\|\bsx \|_p)$ for some function $g:[0,\infty) \to [0,\infty)$. The function $g$ is called the density generator. We can now consider the polar decomposition of $\bsx$ w.r.t.\ the $L_p$-norm and define
$$R := \|\bsX\|_p \quad  \hbox{and} \quad   \bsU_n := \bsX / \|\bsX\|_p$$
as the radial and the angular part, respectively.
In Lemma 2.1 of \citet{gupta-song-97}, it is  shown that 
$R$ and $\bsU_n$ are then independent. The distribution of $\bsU_n$ is called the $L_p$-uniform distribution on the sphere and has been considered in detail in \citet{song-gupta-97}. The general definition of an 
$L_p$-norm spherical distribution is then given by the stochastic representation that $\bsX = R \cdot \bsU_n$ for a radius $R \ge 0$ with an arbitrary univariate non-negative random variable $R$, independent of $\bsU_n$. 

For a random vector $\bsX = (X_1,\ldots,X_n)$ with an $L_p$-norm spherical distribution, define the vector
\begin{equation} \label{x+def}
\bsX^+ := (|X_1|,\ldots,|X_n|)
\end{equation} 
with values in $\R_+^n := [0,\infty)^n$. If $\bsX$ has the density $f(\bsx) = g(\|\bsx \|_p)$, then it is easy to see because of the symmetry that $\bsX^+$ has the density
$$
f^+(\bsx) = 2^n \cdot g(\|\bsx \|_p), \quad \bsx \in \R_+^n.
$$
The distribution of $\bsX^+$ is an $L_p$-norm spherical distribution on $\R_+^n$. Given such a vector $\bsX^+ = (X_1^+,\ldots,X_n^+)$ we can then recover an $L_p$-norm spherical distribution on $\R^n$ with density generator $g$ by the construction
\begin{equation} \label{plus-backtransform}
\bsX =_d (X_1^+ \cdot Z_1,\ldots,X_n^+ \cdot Z_n),
\end{equation}
where $Z_1,\ldots, Z_n$ are i.i.d.~random variables with $P(Z_i = 1) = P(Z_i=-1) = 1/2$, independent of $\bsX^+$. 

\subsection{$L_p$-norm spherical copulas}

In this paper we are interested in random vectors $\bsX$ with $L_p$-norm spherical distributions with uniform marginals. We use the notation $U(a,b)$ for a uniform distribution of a random variable $U$ with density
$$
f_U(x) = \frac{1}{b-a} \one_{[a < x < b]}
$$ 
for $a < b$, where $\one_A$ denotes the indicator function of a set $A$ in contrast to $\bsone = (1,\ldots,1)$ which will henceforth be used for a vector whose components are all equal to 1.
For a random vector $\bsX$ with an $L_p$-norm spherical distribution with $U(-1,1)$ uniform marginals 
we have $|X_i| \sim U(0,1)$ and thus the distribution of $\bsX^+$ is a copula
with copula density 
\begin{equation} \label{lpc-def}
c^+(\bsx) = 2^n \cdot g(\|\bsx \|_p) \cdot \one_{[\|\bsx \|_p <1]} \cdot \one_{[\bsx \in (0,1)^n]}
, 
\end{equation} 
in case such a density exists. We will call such a copula a positive $L_p$-norm spherical copula. Notice that this means that we require the distribution to be $L_p$-norm spherical if considered as a distribution
on $\R_+^n$. This condition is stronger than requiring the existence of a density of the form
$$
c(\bsx) =  g(\|\bsx \|_p) \cdot \one_{[\bsx \in (0,1)^n]}
$$
as Equation \eqref{lpc-def} also implies that $c^+(\bsx )=0$ for all $\bsx \in (0,1)^n$ with
$\|\bsx\|_p \ge 1$.  

An important example showing the difference is the independence copula with density
$$
c(\bsx) =   \one_{[\bsx \in (0,1)^n]}.
$$
This of course depends only on the norm for any $p \in [1,\infty]$, when considered as a function on $(0,1)^n$, but -- according to  \eqref{lpc-def} -- we will call this an $L_p$-norm spherical copula with respect to $L_\infty$ only. 
Using the stronger definition in \eqref{lpc-def} has the advantage that we will get uniqueness of the copulas. Therefore we will use the following formal definition.

\begin{definition} \label{plp-copula}
A copula is called a \textit {positive $L_p$-norm spherical copula}
 if the corresponding distribution is an $L_p$-norm spherical distribution on $\R_+^n$.
\end{definition}

There is another natural copula derived from an  $L_p$-norm spherical distribution of a random vector $\bsX$ with $U(-1,1)$ uniform distributed marginals obtained from the transformation
\begin{equation} \label{lpcirc-def}
\bsX^\circ := \left( \frac{1}{2} (X_1+1), \ldots,  \frac{1}{2} (X_n +1) \right). 
\end{equation} 
In case of an existing density this has the form
$$
c^\circ(\bsx) = 2^n \cdot g(\|2\bsx -   \bsone\|_p) \cdot \one_{[\|2\bsx -\bsone\|_p < \frac{1}{2}]}.
$$
Of course this transformation is a bijection. For any such $\bsX^\circ$ we can recover the corresponding $\bsX$ by considering $X_i := 2 X_i^\circ -1, i=1,\ldots,n$. We thus get the following formal definition of a 
second variant of an $L_p$-norm spherical copula. 

\begin{definition} \label{clp-copula}
The distribution $C^\circ$ of a random vector $\bsX^\circ = (X_1^\circ, \ldots, X_n^\circ)$ with $U(0,1)$ uniform marginals  is called a \textit {circular $L_p$-norm spherical copula}, if the
distribution of the random vector $\bsX = (2 X_1^\circ -1, \ldots, 2 X_n^\circ -1)$ is an $L_p$-norm spherical distribution on $\R^n$.
\end{definition}

Thus we can define for any random vector $\bsX$ with $U(-1,1)$ uniform marginals two versions of $L_p$-norm spherical copulas as the distributions of the transformations $\bsX^+$ and $\bsX^\circ$ 
and there is an obvious bijection between these objects, so that it is sufficient to treat one of these three objects. Very often it is most convenient to work with the version of $\bsX^+$, as one can then avoid
the use of absolute values in the formulas for densities and it additionally has the advantage that the marginal distributions of the corresponding $L_p$-norm uniform distribution can be related to a Beta distribution.

From now on, we will use the notation $Beta(\alpha, \beta)$ for a Beta distribution with the density
$$
f(x) = \frac{\Gamma(\alpha + \beta)}{\Gamma(\alpha) \Gamma(\beta)} \ x^{\alpha -1} (1-x)^{\beta -1}, \quad 0 < x < 1.
$$
This is well defined for $\alpha, \beta > 0$ and obviously we get a uniform distribution $U(0,1)$ if $\alpha = \beta = 1$. The following lemma on properties of the Beta distribution will be used later.

\begin{lemma} \label{lemma-beta}
a) if $X \sim Beta(\alpha,1)$ then $X^s \sim Beta(\alpha / s, 1)$ for all $s > 0$.\\
b) if $X \sim Beta(\alpha, \beta)$ and $Y \sim Beta(\alpha+ \beta, \gamma)$ are independent, then $X \cdot Y \sim Beta(\alpha, \beta + \gamma)$.
\end{lemma}

\begin{proof}
a) follows by a simple calculation and b) can be shown by showing that all moments of $X \cdot Y$ coincide with the moments of a $Beta(\alpha, \beta + \gamma)$ distribution, see \citet{mckenzie-85}.
\end{proof}

For $\bsU_n$ having an $L_p$-uniform distribution on the sphere the marginal distribution of $|U_i|$ is determined in \citet{song-gupta-97}, Theorem 2.1.

\begin{theorem} \label{lpu-marginal}
If $\bsU_n$ has an $L_p$-uniform distribution on the sphere for some $p \in [1,\infty)$, then the marginals fulfill 
$$
|U_i|^p \sim Beta\left(\frac{1}{p},\frac{n-1}{p}\right).
$$
\end{theorem}

\subsection{Existence and uniqueness of 
$L_p$-norm spherical copulas}

For a general $L_p$-norm spherical distribution we can easily show that there is at most one possible distribution with given marginals. The following theorem is a straightforward generalization of Proposition 2.1 in \citet{perlman-wellner-11} and is not really new, but for completeness we add the simple proof, see also Section 4.9 of \cite{joe1997}.

\begin{theorem} \label{th:unique}
$L_p$-norm spherical copulas are unique, if they exist.
\end{theorem}

\begin{proof}
We consider the case of a random vector $\bsX^+ = (X_1,\ldots,X_n)$ with a positive $L_p$-norm spherical copula. According to Theorem \ref{lpu-marginal} the marginals of the positive $L_p$-norm uniform distribution $\bsU_n^+ = (U_1,\ldots,U_n)$ have a transformed Beta distribution, i.e. $U_i^p \sim Beta\left(\frac{1}{p},\frac{n-1}{p}\right)$. From the representation $\bsX^+ = R \cdot \bsU_n^+$ we get for the marginals of $\bsX^+$ the identity $X_i = R\cdot U_i$ and thus
$$
\log(X_i^p) = \log(R^p) + \log(U_i^p).
$$
Because of the independence of $R$ and $U_i$ it follows that the characteristic function of $\log(R^p)$ must be the ratio of the two characteristic functions of $\log(X_i^p)$ and $\log(U_i^p)$, which both have a known distribution. Thus the distribution of  $\log(R^p)$ and therefore also the distribution of $R$ is uniquely determined. 
\end{proof}

We also easily get the following necessary condition for the existence.

\begin{theorem} \label{th:necessary}
$L_p$-norm spherical copulas can only exist for $p \ge n-1$.
\end{theorem}

\begin{proof}
Analogously to Theorem \ref{th:unique} we use the representation $X_i = R\cdot U_i$. We know that  $U_i^p \sim Beta\left(1/p,(n-1)/p\right)$ and that $X_i \sim U(0,1)$ and hence $X_i^p \sim Beta(1/p,1)$. 
The mean of a $Beta(\alpha,\beta)$ distribution is given by $\alpha/(\alpha+\beta)$. From the independence of $R$ and $U_i$ we thus derive from the identity $E(X_i^p) = E(R^p) \cdot E(U_i^p)$
the equation
$$
\frac{1}{p+1} = E(R^p) \cdot \frac{1}{n}.
$$
As necessarily $R \le 1$ and thus also $E(R^p) \le 1$ we get the necessary condition $p \ge n-1$.
\end{proof}
We will see that this condition is also sufficient to get an $L_p$-norm spherical copula in Section~\ref{sec:multivariate}.

\section{Bivariate $L_p$-norm spherical copulas} \label{sec:bivariate}

We first consider the bivariate case where we can generalize the results of \citet{perlman-wellner-11} from the case $p=2$ to arbitrary $p \in [1,\infty)$. 
\citet{perlman-wellner-11} considered the question whether there exists an $L_2$-norm spherical distribution with marginals that are uniform on $[-1,1]$. The answer is affirmative in dimensions $n = 2$ and $n=3$. In $\R^3$ the corresponding distribution is just the uniform distribution $\bsU_3$ on the sphere. In dimension $n=2$ the distribution has the density
$$
f(x,y) = \frac{1}{2 \pi \sqrt{1-x^2-y^2}} \one_{[x^2+y^2<1]},
$$
i.e.\ the corresponding positive $L_2$-norm spherical distributions with $U(0,1)$ uniform marginal distributions are the uniform distribution $\bsU_3^+$ and the distribution
with density 
$$c^+(x,y) = \frac{2}{\pi \sqrt{1-x^2-y^2}} \one_{[x^2+y^2<1]} \one_{[x > 0,y > 0]},
$$
for $n=3$ and $n=2$, respectively. We generalize these results for $p>1$ as follows.

\begin{theorem} \label{th1}When $n=2$, the density of the $L_p$-norm spherical copula for $p>1$ is given by 
$$
f(x,y) = \frac{1}{\Gamma(1-1/p) \Gamma(1+1/p) (1-x^p-y^p)^{1/p}} \one_{[x^p+y^p<1]},
$$
for $x>0$ and $y>0$.
\end{theorem}

\begin{proof}
To generalize the result of \citet{perlman-wellner-11} for $p=2$  to the case of arbitrary $p > 1$, we  mimick the proof given there. Thus we assume that the density has the form $c^+(x,y) = 4 g((|x|^p+|y|^p)^{1/p})$ for $|x|^p+|y|^p <1$ as in \eqref{lpc-def}. Using first the transformations $h(t) = g((1-t)^{1/p})$ and then the change of variable $u = y/(1-x^p)^{1/p}$
we get  the marginal distribution $c^+_1(x)$ by integrating over $y \in (0,1)$ for $x > 0,$
\begin{align*}
c^+_1(x) & = 4 \int_{0}^{1} g((x^p+y^p)^{1/p}) \one_{[x^p+y^p<1]}\dd y\\
& = 4  \int_0^{(1-x^p)^{1/p}} h(1-x^p-y^p) \dd y\\
& = 4  \int_0^1 h(1-x^p-u^p(1-x^p)) (1-x^p)^{1/p} \dd u\\
& = 4 (1-x^p)^{1/p}  \int_0^1 h((1-u^p)(1-x^p)) \dd u.
\end{align*}
Assuming that $h(t) = c \, t^{-1/p}$ for some constant $c > 0$ and transforming $z = u^p$, we get
\begin{align*}
c^+_1(x) ={}& 4c \,  \int_0^1 (1-u^p)^{-1/p} \dd u
= 4 c   \int_0^1 (1-z)^{-1/p} \frac 1 p z^{\frac 1 p -1} \dd z\\
={}& 4 c \frac{ \Gamma(1-1/p) \Gamma(1/p) }{p} = 4 c \Gamma(1-1/p) \Gamma(1+1/p) .
\end{align*}

We thus find the appropriate constant $c$ to obtain a uniform distribution on $[0,1]$ for the marginals and 
$$g\left(t^{1/p}\right) = h(1-t) =  \frac{1}{4  \Gamma(1-1/p) \Gamma(1+1/p)}\, (1-t)^{-1/p}$$
 and thus for the density
$$
c^+(x,y) = 4 g(\|(x,y) \|_p) = \frac{1}{\Gamma(1-1/p) \Gamma(1+1/p) (1-x^p-y^p)^{1/p}}, \quad x,y > 0, \ x^p + y^p < 1.
$$
\end{proof}

\subsection{Some properties of the $L_p$-norm spherical copula when $n=2$ }
From Theorem \ref{th1}, and using a result from \citet{gupta-song-97}, we can also deduce the distribution of the radius $R$, which has the p.d.f.
$$
f_R(r) = \frac{4\cdot \Gamma(1/p)^2}{p \cdot \Gamma(2/p)} \cdot r \cdot g(r^p) = \frac{\Gamma(1/p)}{ \Gamma(2/p)\Gamma(1-1/p)} \cdot r \cdot (1-r^p)^{-1/p}.
$$
From a simple density transformation, one can easily see that $R^p$ has a Beta distribution with parameters $2/p$ and $1-1/p$.  Thus, for the $p$-th order moment of $R$, we obtain
$$ E(R^p) = \frac 2 {p+1},$$
which is a decreasing function of $p$ with a limit of $E(R^p) = 1$ for $p \to 1$. This means that the distribution gets more and more concentrated at the boundary of the unit ball if $p \to 1$.
Indeed, in the limit $p =1$ we get as distribution the uniform distribution on the set $\{(x,y) \in [0,1]^2: |x| + |y| = 1\},$ which obviously has uniform marginals and is the well-known lower Fr\'echet bound
with 
$$
\bsX^+ = (X_1^+,X_2^+) = (U,1-U)
$$
for some uniform $U \sim U(0,1)$ with correlation coefficient $\rho=-1$. In the limit $p \to \infty$ we approximate a uniform distribution on the square, which can be considered as an $L_\infty$ symmetric copula as discussed in Section \ref{S4}. Indeed, we thus get a family of copulas, which interpolates between independence and complete negative dependence. We can compute explicitly the correlation coefficient.

\begin{theorem}
If $\bsX^+ = (X_1^+,X_2^+)$ has a positive $L_p$-norm symmetric copula, then its correlation coefficient is given by
\begin{equation}\label{corr}
\rho_p =  \frac{4 \Gamma(2/p)^2}{\Gamma(1/p) \cdot  \Gamma(3/p)} - 3
\end{equation}
with $\lim_{p\to1} \rho_p = -1$, $\lim_{p\to\infty} \rho_p = 0$ and $p \mapsto \rho_p$ increasing.
\end{theorem}

\begin{proof}
Using the transformation $x = u \cdot (1-y^p)^{1/p}$ we get for the expectation $E(X_1^+ \cdot X_2^+)$ the expression
\begin{align*}
\int xy c^+(x,y) \dd x \dd y & =  \frac{1}{\Gamma(1-1/p) \Gamma(1+1/p)} \int_0^1  \int_0^{(1-y^p)^{1/p}}  \frac{xy}{(1-x^p-y^p)^{1/p}} \dd x \dd y\\
& = \frac{1}{\Gamma(1-1/p) \Gamma(1+1/p)} \int_0^1 \int_0^1  \frac{uy(1-y^p)^{1/p}}{((1-y^p)(1-u^p))^{1/p}} (1-y^p)^{1/p} \dd u \dd y\\
& = \frac{1}{\Gamma(1-1/p) \Gamma(1+1/p)} \int_0^1 y  (1-y^p)^{1/p}  \dd y  \int_0^1  u (1-u^p)^{-1/p} \dd u\\
& = \frac{1}{\Gamma(1-1/p) \Gamma(1+1/p)} \cdot \frac{1}{p} \frac{ \Gamma(1+1/p) \Gamma(2/p) }{\Gamma(1+3/p)} \cdot \frac{1}{p} \frac{ \Gamma(1-1/p) \Gamma(2/p) }{\Gamma(1+1/p)} \\
& =  \frac{ \Gamma(2/p)^2}{\Gamma(1/p) \cdot 3 \cdot \Gamma(3/p)}.
\end{align*}
As $E(X_i^+) = 1/2$ and $Var(X_i^+) = 1/12$ we have
$$
\rho_p = \frac{E(X_1^+ \cdot X_2^+) - E(X_1^+) E(X_2^+)}{Var(X_1^+)} = 12 E(X_1^+ \cdot X_2^+) - 3,
$$
which yields equation \eqref{corr}. Inserting $p = 1$ we immediately get $\rho_1 = -1$ and using $\Gamma(x) = \Gamma(x+1)/x$ we derive
$$
\lim_{p\to\infty} \rho_p = \lim_{p\to\infty} \frac{\Gamma(1+ 2/p)^2}{\Gamma(1+1/p) \cdot  \Gamma(1+ 3/p)/3} - 3 = 0.
$$
Finally, we show that $p \mapsto \rho_p$ is increasing, or equivalently that
$$
x \mapsto  f(x) := \frac{\Gamma(2x)^2}{\Gamma(x) \cdot  \Gamma(3x)}
$$
is decreasing for $0 < x < 1$. Taking the derivative of the logarithm we get
$$
g(x) := \frac{\dd}{ \dd x} \log(f(x)) = 4 \psi(2x) - \psi(x) - 3\psi(3x),
$$
where $\psi(x) = \Gamma'(x)/\Gamma(x)$ is the so-called Digamma function. Thus it is sufficient to show that this function is negative. To do so, we use the following representation 
of the Digamma function $\psi$ that can be found as formula 6.3.16 in \cite{abramowitz2006}. We have
$$
\psi(x) = - \gamma + \sum_{n=0}^\infty \left(  \frac{1}{n+1} -   \frac{1}{n+x}  \right),
$$
where $\gamma$ is the Euler-Mascheroni constant. Thus we get
\begin{align*}
g(x) & = - \sum_{n=0}^\infty \left(  \frac{4}{n+2x} -   \frac{1}{n+x} -  \frac{3}{n+3x}\right)\\
& = - \sum_{n=0}^\infty  \frac{2nx}{(n+2x)(n+x)(n+3x)} < 0, \qquad \text{for all } 0 < x < 1.
\end{align*}
\end{proof}

\subsection{A Conjecture}

As announced in the introduction, we conjecture that the bivariate $L_p$-norm spherical copulas play an important role in the problem of finding sharp dependence uncertainty bounds for the generalized energy distance. Specifically, we consider the following optimization problem:
\begin{equation}\label{MAXsco}\max_{C\in\CC} S_\beta(C,C)\end{equation}
where $\CC$ denotes the set of all possible bivariate copulas and where we recall that $S_\beta(C,C)$ denotes the  generalized expected distance between two independent samples $\bsX$ and $\bsY$ of the copula  $C$ defined formally as $
S_\beta(C,C) = \E\left( \| \bsX - \bsY \|^\beta_2 \right).$  In the case $\beta=1$, \cite{BM2020} show that the bivariate spherical copula $C$ is the solution to  \eqref{MAXsco}. In the general case of $\beta \neq 1$, we do not have a solution of the problem. However, our numerical experiments led us to formulate the following conjecture.

\begin{conjecture} The solution to \eqref{MAXsco} is obtained for the circular bivariate $L_p$-norm spherical copula with $p=3-\beta$ for arbitrary $\beta \in (0,2)$.
\end{conjecture}

In what follows, we illustrate why we believe that the conjecture is true. We solve  \eqref{MAXsco} heuristically, by first discretizing the marginal distributions with 1,000 discretization points and then by using a modified swapping algorithm (in the spirit of \citealp{puccetti2017} and \citealp{puccetti2020computation}) to approximate the solution of  \eqref{MAXsco}.

Starting from a random permutation, we display in Figure \ref{circleU01} the initial permutation and the final permutation (obtained in about 10 minutes) that is such that any swaps of two elements does not result in a strict increase in the objective function. The result for the case $\beta =1$ is displayed in Figure 1. The heuristically obtained solution looks pretty much like a bivariate projection of a uniform distribution on a ball and in the case $\beta =1$ it is proved in \cite{BM2020} that this indeed solves the optimization problem  \eqref{MAXsco}. 

\begin{figure}[!htbp] 
\begin{center}
\includegraphics[width=\textwidth,height=0.4\textwidth]{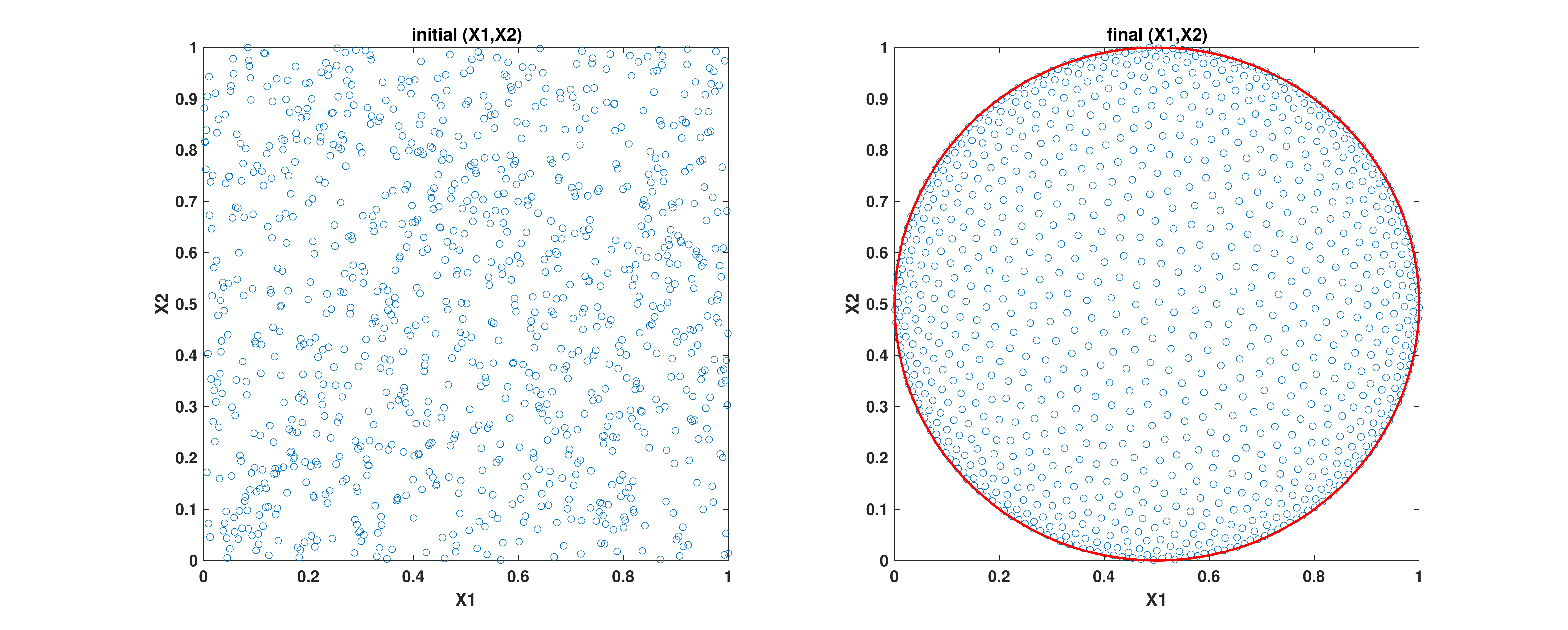}
\end{center}
\caption{Two-dimensional illustration of the Swapping Algorithm for $S_1(C,C)$ where each marginal is discretized with 1,000 points of discretization. In the right panel, a red circle with radius 0.5 and center (0.5,0.5) is also displayed. \label{circleU01}}
\end{figure}

Repeating the above experiment with $\beta=0.2$, $\beta=0.5$, $\beta=1.5$ and $\beta=1.8$, we obtain four other copulas, showing that the maximizing copula in \eqref{MAXsco} depends on $\beta$.  A red line denotes the boundary of the support of the symmetric $L_{p}$-norm spherical copula for $p=3-\beta$.  See Figure \ref{circleU01beta}.

\begin{figure}[!htbp]
\begin{center}
\begin{tabular}{cc}
\includegraphics[width=0.5\textwidth,height=0.4\textwidth]{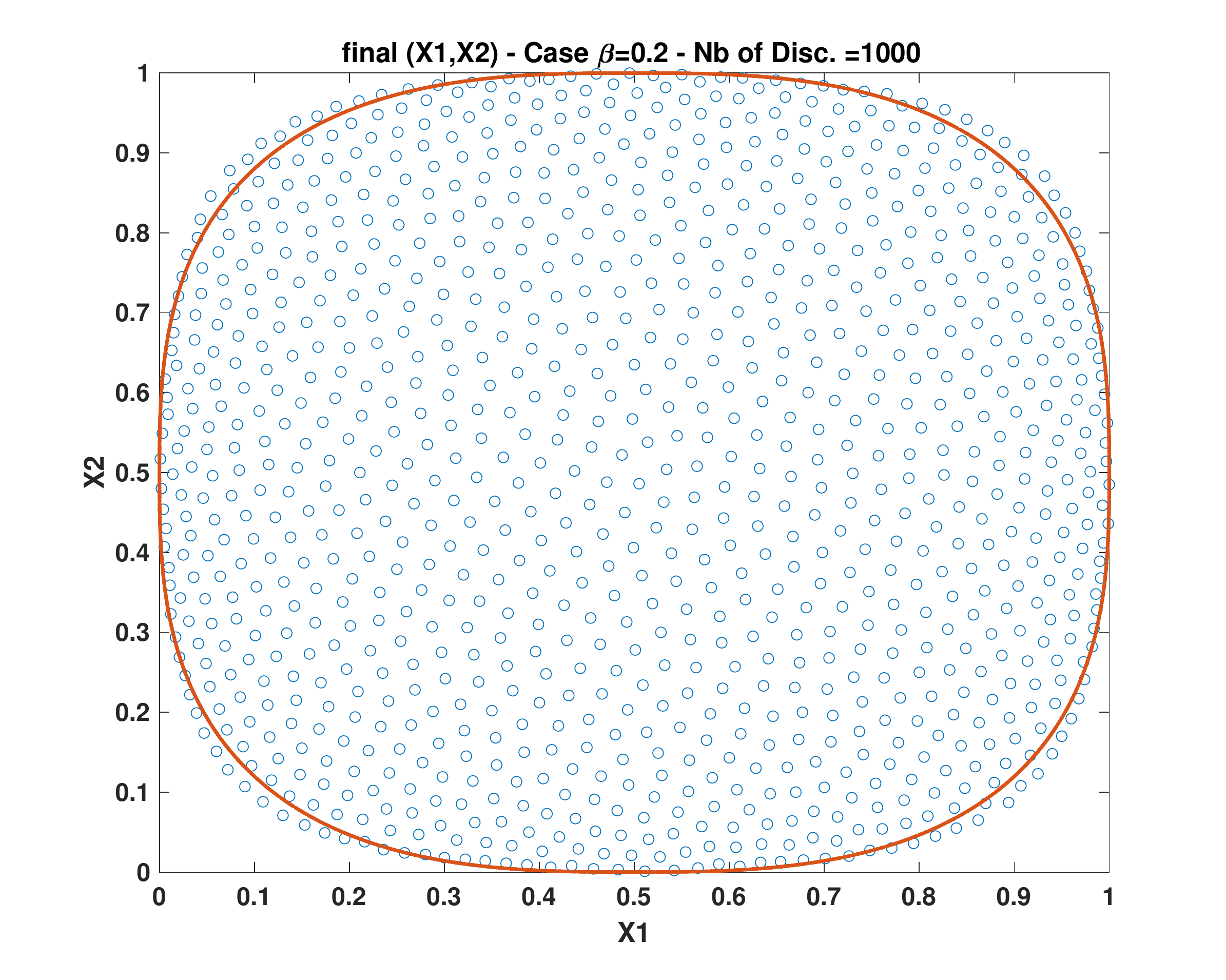}&\includegraphics[width=0.5\textwidth,height=0.4\textwidth]{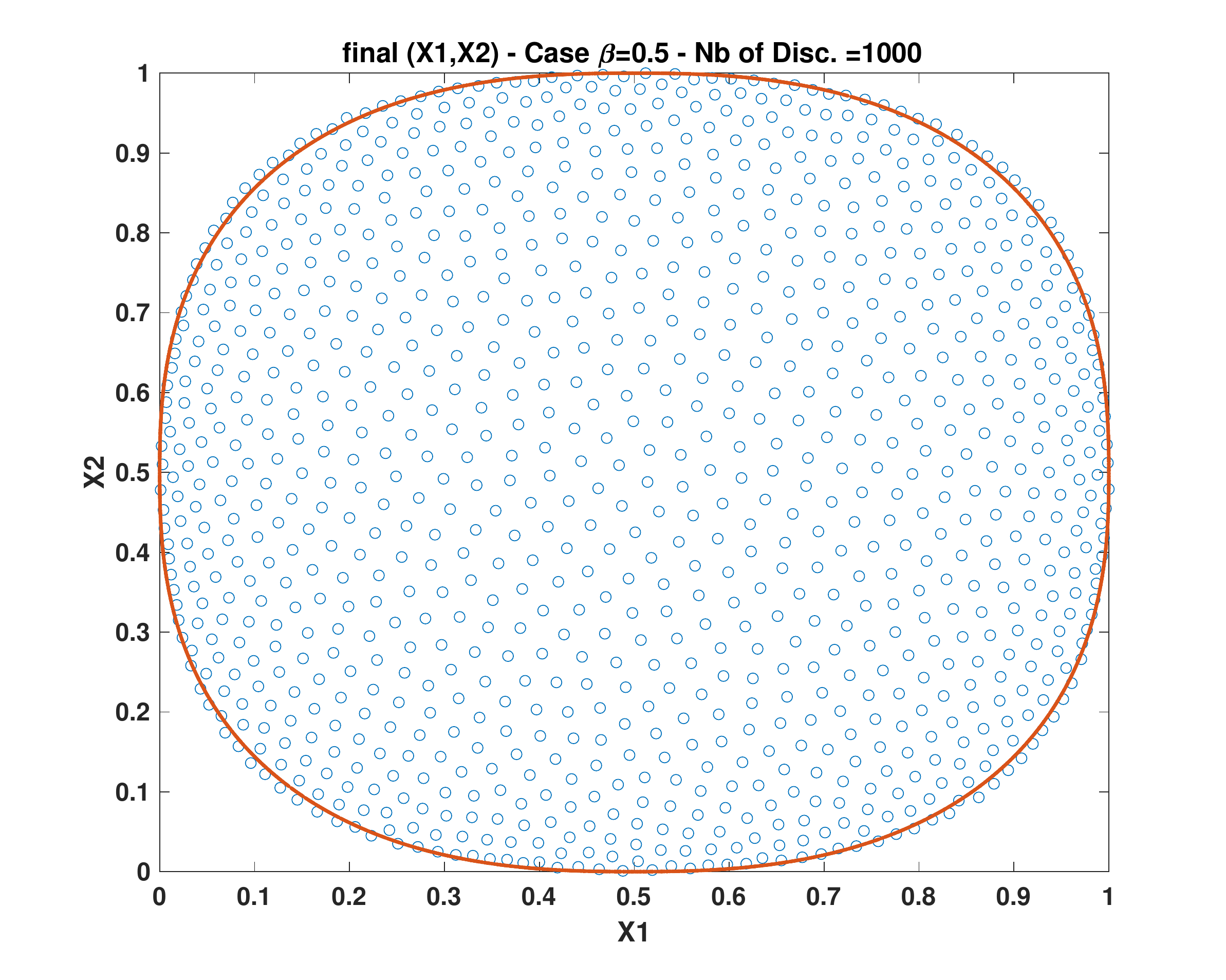}
\\
$\beta=0.2$  &$\beta=0.5$ \\
\includegraphics[width=0.5\textwidth,height=0.4\textwidth]{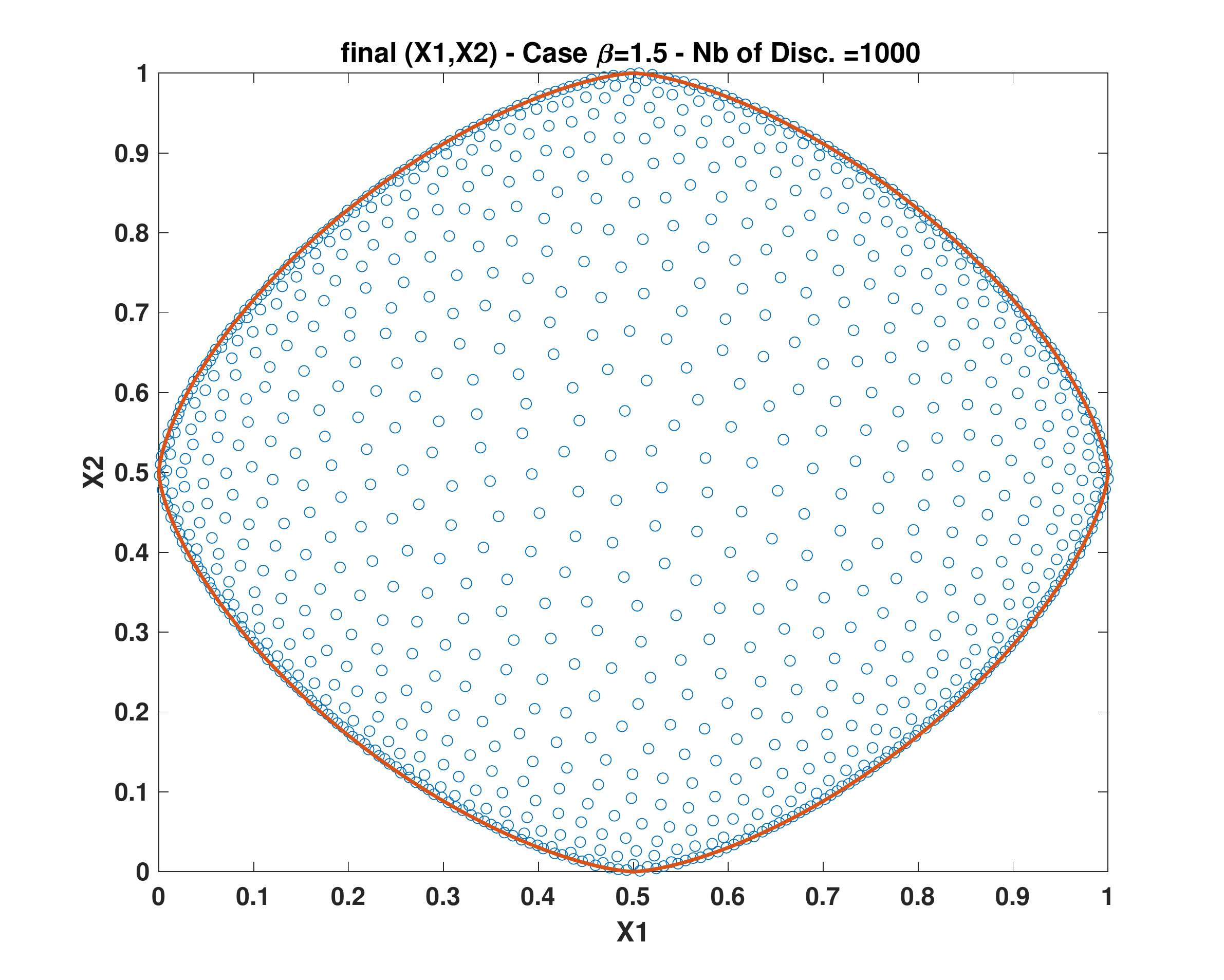}&\includegraphics[width=0.5\textwidth,height=0.4\textwidth]{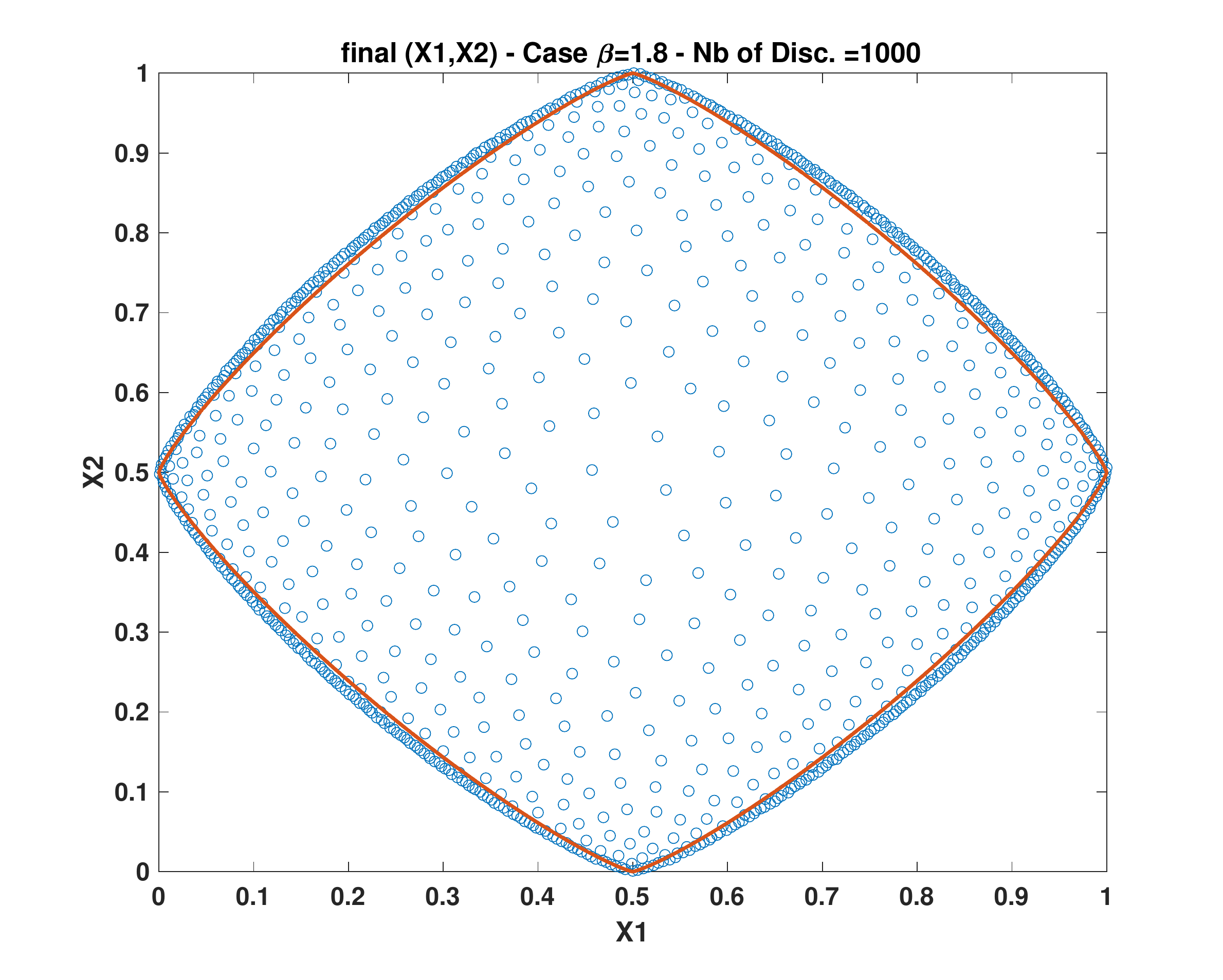}\\
$\beta=1.5$ &$\beta=1.8$\\
\end{tabular}

\end{center}
\caption{Two-dimensional illustration of the Swapping Algorithm for $S_\beta(C,C)$ with 1,000 points of discretization. A red bowl with radius 0.5 and center at $(0.5,0.5)$ for the $L_{p}$-norm with $p=3-\beta$ is also displayed.\label{circleU01beta}}
\end{figure}

\newpage

\section{Multivariate $L_p$-norm spherical copulas}\label{sec:multivariate}

We will now give a general characterization of $L_p$-norm spherical copulas in arbitrary dimension $n \ge 2$ and any $1 \le p < \infty$. We have already seen in Theorem \ref{th:necessary} that it is necessary that $p \ge n-1$. Similarly as in the bivariate case discussed in Section \ref{sec:bivariate}, we will now see that this is also sufficient to obtain a positive $L_p$-norm spherical copula with a radius following a transformed Beta distribution.


\begin{theorem} \label{thm:lps-multi}
For any $n \ge 2$ and $p \ge n-1$ there exists a unique positive $L_p$-norm spherical copula of a random vector $\bsX^+ = R \cdot \bsU_n^+$. If $p > n-1$ then the radius $R$ has the property that $R^p \sim Beta(\frac{n}{p}, 1- \frac{n-1}{p})$ and hence $R$ has a density of the form
$$
f_R(r) = \frac{\Gamma(1/p)}{\Gamma(n/p) \Gamma(1 - \frac{n-1}p)}  \cdot r^{n-1} \cdot (1-r^p)^{-\frac{n-1}{p}} \one_{[0<r<1]}.
$$
If $p = n-1$ then the unique positive $L_p$-norm spherical copula has the property $R = 1$ a.s. and thus $\bsX^+ = \bsU_n^+$.
\end{theorem}

\begin{proof}
Let $\bsU_n^+ = (U_1,\ldots,U_n)$ be a random vector uniformly distributed on the set $\{\bsx \in \R^n: \|\bsx\|_p = 1, \bsx \ge \bsnull \}$. 
According to Theorem \ref{lpu-marginal} we have $U_i^p \sim Beta\left(\frac{1}{p},\frac{n-1}{p}\right)$. Assume that $R^p \sim Beta(\frac{n}{p}, 1- \frac{n-1}{p})$ and is independent of $U_i$. It follows from Lemma \ref{lemma-beta} b) 
that $R^p \cdot U_i^p \sim Beta\left(\frac{1}{p},1\right)$ and thus from Lemma \ref{lemma-beta} a) that $R \cdot U_i = (R^p \cdot U_i^p)^{1/p} \sim Beta(1,1) = U(0,1)$. Thus 
$\bsX^+ = R \cdot \bsU_n^+$ has $U(0,1)$ marginals and therefore its distribution is a copula. The uniqueness follows from Theorem \ref{th:unique}. 

To derive the density of $R$, notice that if $Y = R^p \sim Beta(\frac{n}{p}, 1- \frac{n-1}{p})$ then $P(R \le x) = P(Y \le x^p)$ and therefore
$$
f_R(r) = f_Y(r^p) \cdot p \cdot r^{p-1}.
$$
Inserting for $f_Y$ the density of the Beta distribution yields
$$
f_R(r) = \frac{p \Gamma(1+\frac 1 p)}{\Gamma(n/p) \Gamma(1 - \frac{n-1}p)} \cdot (r^{p})^{\frac{n}{p}-1} (1-r^p)^{- \frac{n-1}{p}} \cdot r^{p-1} = \frac{\Gamma(1/p)}{\Gamma(n/p) \Gamma(1 - \frac{n-1}p)}  \cdot r^{n-1} \cdot (1-r^p)^{-\frac{n-1}{p}}
$$
for $0 < r < 1$ and of course $f_R(r) = 0$ for $r \not\in (0,1)$.
If $p = n-1$ then the derivation simplifies to $U_i^p \sim Beta\left(\frac{1}{p},1\right)$ and thus $U_i \sim U(0,1)$, so that $R = 1$ a.s. yields uniform marginals for $R \cdot U_i$.
\end{proof}
\medskip

From results by \citet{song-gupta-97} and a density transform, we can compute the density of the positive $L_p$-norm spherical copula.

\begin{corollary}
	Let $n \geq 2$ and  $p\geq n-1$. Then, the probability density of the unique positive
	$L_p$-norm spherical copula is given by
	\begin{equation} \label{eq:copula-density}
	 c^+(\bsx) = c^+(\bsx; p) = \frac{1}{\Gamma^{n-1}(1 + \frac 1 p) \Gamma(1 - \frac{n-1}{p})} (1-\|\bsx\|_p^p)^{-(n-1)/p}, \quad \bsx \geq \bsnull, \ \|\bsx\|_p < 1.
	\end{equation}
\end{corollary}	
\begin{proof}
	First, we note that $\bsX^+$ can be obtained from the vector $(R,|U_1|,\ldots,|U_{n-1}|)$ by the inverse of the transformation $H(\bsx) =(\|\bsx\|_p, x_1/\|\bsx\|_p,\ldots,x_{n-1}/\|\bsx\|_p)$.
	Using the probability density $f_R$ calculated in Theorem \ref{thm:lps-multi}, the joint density
	$$ f_{|U_1|,\ldots,|U_{n-1}|}(u_1,\ldots,u_{n-1}) = \frac{p^{n-1} \Gamma(n/p)}{\Gamma^n(1/p)} \left( 1 - \sum_{i=1}^{n-1} |u_i|^p\right)^{(1-p)/p}$$
	obtained in Theorem 1.1 in \citet{song-gupta-97} and the determinant of the Jacobian
	$J_H,$ which is equal to
	$$ \det(J_H(\bm x)) = \|x\|_p^{1-n} \left(1 - \sum_{i=1}^{n-1} \frac{x_i^p}{\|\bsx\|_p^p}\right)^{1-1/p}$$
	according to Lemma 1.1 in \citet{song-gupta-97}, the result directly follows from
	the density transformation formula
	$$ c^+(\bsx) = f_{R}(\|\bsx\|_p) \cdot
	               f_{|U_1|,\ldots,|U_{n-1}|}\left(\frac{x_1}{\|\bsx\|_p^p}, \ldots, \frac{x_{n-1}}{\|\bsx\|_p^p}\right) \cdot
	               |\det(J_H(\bsx))|.$$
	               
\end{proof}

\cite{gupta-song-97} also show in their Theorem 4.1 that subsets of $L_p$-norm spherical distributions are still $L_p$-norm spherical distributions. Therefore, we can consider the opposite question of extendibility. This topic has seen increasing interest recently, see \cite{takis2019} and \cite{mai2020} for some recent expositions on the general question of finite and infinite extendibility of exchangeable random vectors. From the previous discussion we get here immediately the following result showing that we only have finite extendibility here, depending on the parameter $p$. 

\begin{theorem}
A random vector $\bsX^+ = (X_1,\ldots,X_d)$ with uniform marginals and an $L_p$-norm spherical copula can be extended to a random vector $(X_1,\ldots,X_n)$ with $n > d$ having an $L_p$-norm spherical copula if and only if $n \le p+1$. 
\end{theorem}

\paragraph{Simulation.}
We can easily generate samples from the positive $L_p$-norm spherical copula 
using the stochastic representation $\bsX^+ = R \cdot \bsU_n^+$. 
To this end, we start from an arbitrary $L_p$-spherical symmetric random vector $\bsZ$, e.g.~the random vector $\bsZ$ with i.i.d.~components which have a $p$-generalized normal distribution with density
$$
f(x) = \frac{p}{2^{1+1/p} {\Gamma(1/p)}} \cdot \exp(-{|x|_p^p/2}), \quad x \in \R, 
$$ 
see Example 2.2 in \citet{gupta-song-97} with $r = 1/2$.
From this we can sample realizations of $\bsU_n^+$ by using the formula 
$$ \bsU_n^+ := \frac{1}{\|\bsZ\|_p}(|Z_1|,\ldots,|Z_n|)^\top.$$
Then we have to multiply this vector with the independent random radius $R$ given in Theorem \ref{thm:lps-multi}, i.e.\ the $p$-th root of a Beta distributed vector, to get the random vector $\bsX^+$ with uniform marginals.
Examples of such samples can be found in Figure \ref{fig:samples}.

\begin{figure}
	\includegraphics[width=15cm]{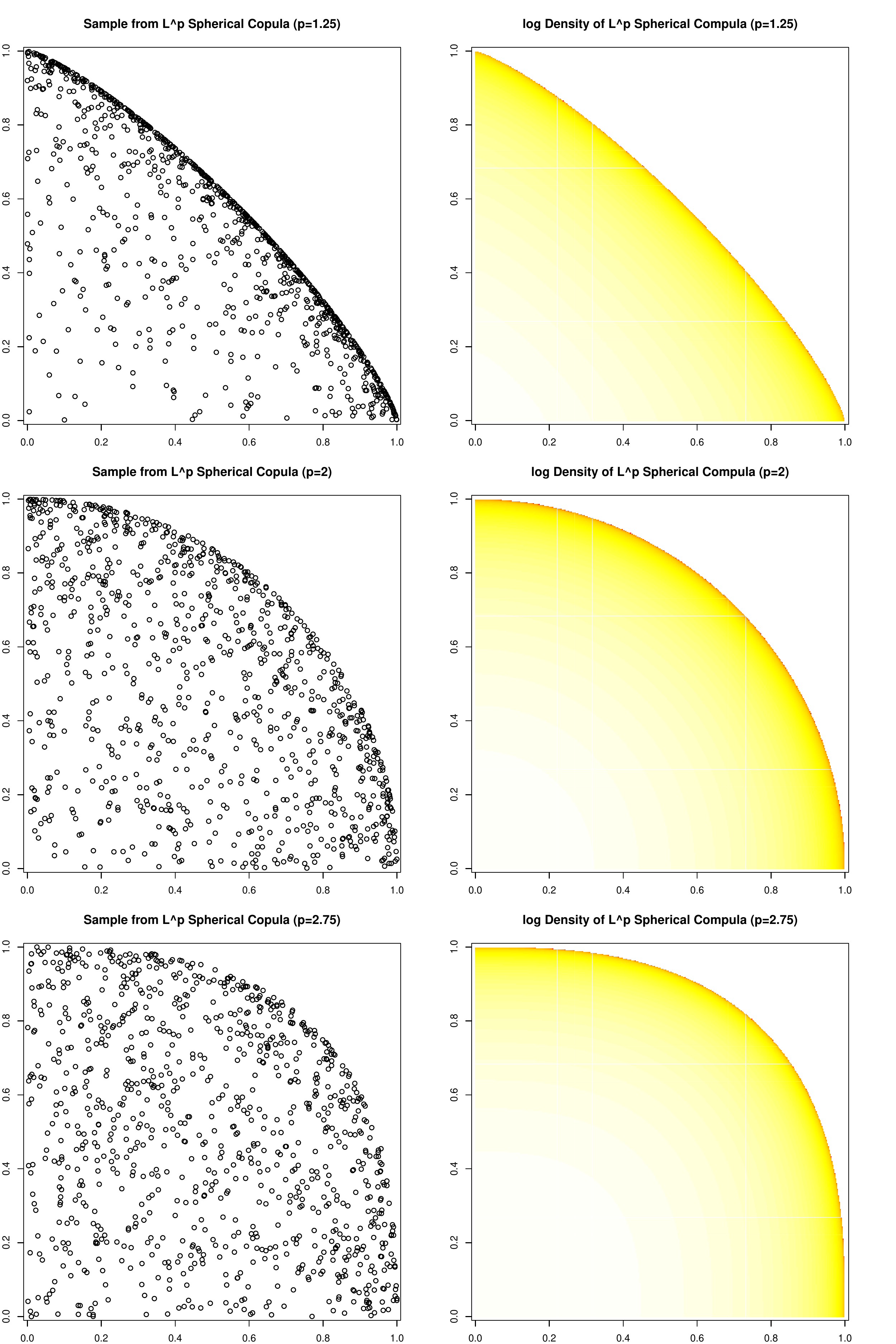}	
  \caption{Left: Samples of size 1,000 from the positive  $L_p$-norm spherical copulas with $p=1.25$, $2$ and $2.75$ (top to bottom), respectively. Right: The corresponding probability density functions.}  \label{fig:samples}	
\end{figure}

The simulation of the Beta distributed radius can be avoided if $p$ is an integer.
In this case, the unique $(p+1)$-dimensional positive $L_p$-norm spherical distribution is the $L_p$-uniform distribution. As subvectors of $L_p$-norm spherical random vectors are again $L_p$-norm spherically distributed \citep[cf.~Theorem 4.1 in][]{gupta-song-97}, samples from an $L_p$-norm spherical distribution
in lower dimension can just be obtained by projection.

\paragraph{Inference.} For statistical inference, we consider $m$ independent observations $\bsx^{(1)}, \ldots, \bsx^{(m)}$ from an $n$-dimensional positive $L_p$-norm spherical copula with unknown parameter $p > n-1$. In view of the explicit expression for the probability density given in \eqref{eq:copula-density}, the use of maximum likelihood inference is appealing. In case that all the observations are inside the positive part of the $L_{n-1}$ unit ball $\{\bsx \in [0,1]^n: \, \|\bsx\|_{n-1} < 1, \, \bsx \geq \bsnull\}$, the likelihood remains bounded and can therefore be maximized within the interval $(n-1,\infty)$.

Of particular interest is the case when there is at least one observation in $\{ \bsx \in [0,1]^n: \, \|\bsx\|_{n-1} > 1, \, \bsx \geq \bsnull \}$, as this observation restricts the set of admissible values for $p$ such that $\prod_{i=1}^m c^+(\bsx^{(i)};p) > 0$. More precisely, due to the fact
that the function $p \mapsto \|\bsx\|_p$ is continuous and monotonically decreasing, there is
\begin{align*}
 p^* ={}& \inf\{p > n-1: \ \|\bsx^{(i)}\|_p < 1 \text{ for all } i=1,\ldots,m \} \\
     ={}& \inf\{p > n-1: \ c^+(\bsx^{(i)};p) > 0 \text{ for all } i=1,\ldots,m \} > n-1
\end{align*}       
and we have that $\max_{i=1,\ldots,m} \|\bsx^{(i)}\|_{p^*} = 1$. While, for $\bsx^{(i)}$ with $\|\bsx^{(i)}\|_{p^*} < 1$, the function $p \mapsto c^+(\bsx^{(i)};p)$ is both bounded away from zero and bounded from above uniformly on the interval $(p^*,\infty)$, any $\bsx^{(i)}$ with $\|\bsx^{(i)}\|_{p^*} = 1$ satisfies that $\lim_{p \downarrow p^*} c^+(\bsx^{(i)};p) = \infty$.
Even though the maximum likelihood estimator is not well-defined in this case, $p^*$ can be seen as a natural estimator. Note that this is similar to the nonregular cases considered in \citet{smith-85} where the location parameter is estimated such that the support of the distribution is the minimal set containing all the observations.

\section{$L_\infty$-norm spherical distributions and copulas\label{S4}}

The case of spherical distributions with respect to $L_\infty$-norm has not been considered in \cite{gupta-song-97}. However there is a discussion of infinite exchangeable sequences of $L_\infty$-norm spherical distributions in \cite{gnedin1995}. For convenience we consider here again only the case
of positive $L_\infty$ spherical distributions on $\R_+^n$. It is obvious that $L_\infty$ spherical distributions with a density of the form $f(\bsx) = g(\|\bsx\|_\infty)$ exist, as the case of 
a random vector $\bsU_n^\bot = (U_1^\bot,\ldots,U_n^\bot)$ with i.i.d.~$U_1^\bot,\ldots,U_n^\bot \sim U(0,1)$ is an example with 
$$
f_{\bsU_n^\bot}(\bsx) = g(\|\bsx\|_\infty) = \one_{[\|\bsx\|_\infty < 1, \, \bsx \geq \bsnull]}.
$$
We write $\bsU_n^\bot \sim \mathcal{U}_n(0,1)$ for its distribution.
This indeed is of course a copula and thus we have already derived a positive $L_\infty$ spherical copula. Notice that in this case the corresponding circular $L_\infty$ spherical copula
of the corresponding vector $\bsX^\circ$ is exactly the same.

If we define $\bsU_n^\infty := \bsU_n^\bot / \| \bsU_n^\bot \|_\infty$ and $R :=  \|\bsU_n^\bot\|_\infty$ then we get here also that $R$ and $\bsU_n^\infty$ are independent and it is natural to call
the distribution of $\bsU_n^\infty$ the uniform distribution on the positive part of the $L_\infty$ unit sphere $\{\bsx \in \R^n: \, \| \bsx\|_\infty = 1, \, \bsx \ge \bsnull\}$. The reason why this case has not been considered
in \citet{gupta-song-97} and \citet{song-gupta-97} is probably based on the fact that their definition of an $L_p$ uniform distribution is based on the fact that its lower dimensional marginals have a density
what is not the case here. For the $L_\infty$ uniform distribution we get for the distribution of a subvector a mixture of a uniform distribution inside the ball and uniform distributions on the hypersurfaces of the ball.

Let us denote by $\mathcal{U}_{i:n}(0,1)$ the distribution of the random vector 
$$
\bsU_{i:n} := (U_1, \ldots, U_{i-1}, 1, U_{i+1}, \ldots, U_n)
$$ with i.i.d.~$U_1,\ldots,U_n \sim U(0,1)$, which is a uniform distribution on a hypersurface of the positive part of an $L_\infty$ ball $\{\bsx \in \R^n: \| \bsx\|_\infty = 1, \bsx \ge \bsnull\}$.
Then we can state the following result for subvectors of an $L_\infty$ uniform distribution.

\begin{theorem}
Assume that $\bsU_n^\infty = (U_1^\infty,\ldots,U_n^\infty)$ has an $L_\infty$ uniform distribution. Then for $k < n$ the subvector $(U_1^\infty,\ldots,U_k^\infty)$ has the distribution
$$
(U_1^\infty,\ldots,U_k^\infty) \sim \frac{1}{n} \sum_{i=1}^k \mathcal{U}_{i:k}(0,1) +  \frac{n-k}{n}  \mathcal{U}_k(0,1).
$$
In particular, for the univariate marginals we get
$$
U_i^\infty  \sim \frac{1}{n} \delta_1 +  \frac{n-1}{n}  U(0,1) \quad \mbox{ for all } i = 1,\ldots, n.
$$
\end{theorem}

\begin{proof}
Let $\bsU_n^\bot \sim \mathcal{U}_n(0,1)$ and $\bsU_n^\infty := \bsU_n^\bot / \| \bsU_n^\bot \|_\infty = (U_1^\infty,\ldots,U_n^\infty)$. For $\bsU_n^\bot = (U_1^\bot,\ldots,U_n^\bot)$ with i.i.d.~$U_1^\bot,\ldots,U_n^\bot \sim U(0,1)$ we have a.s.~a unique maximum 
$$ \|\bsU_n^\bot\|_\infty = \max\{U_1^\bot,\ldots,U_n^\bot\} $$
and therefore due to symmetry $P(U_i^\infty = 1) = P(\max\{U_1^\bot,\ldots,U_n^\bot\} = U_i^\bot) =   1/n$ for $i = 1,\ldots,n$ and these events are disjoint. Given $U_i^\infty = 1$, i.e. $\max\{U_1^\bot,\ldots,U_n^\bot\} = U_i^\bot$, the conditional distribution of the other $n-1$ components of $\bsU_n^\infty$ is a uniform distribution $\mathcal{U}_{n-1}(0,1)$. Therefore, the conditional distribution of $(U_1^\infty,\ldots,U_k^\infty)$ given 
$U_i^\infty=1$
for some $i = 1,\ldots, k$ is given by $\mathcal{U}_{i:k}(0,1)$. With probability $ \frac{n-k}{n}$ we have $U_i^\infty < 1$
for all $i = 1,\ldots, k$ and then the conditional distribution under this event is a uniform distribution. 
\end{proof}

From this result, we can easily see that the $L_\infty$-norm spherical distribution can be perceived as weak limit of $L_p$-norm spherical distributions as $p \to \infty$.
To this end, let $R^{(p)}$ and $(U_1^{(p)},\ldots,U_n^{(p)})$ denote the radial and the angular components of a positive  $L_p$-norm spherically distributed random vector. Then, we can make use of Theorem 2.1 in \citet{song-gupta-97} to obtain that, for all $u_1,\ldots,u_k \in [0,1)$,
$k \in \{1,\ldots,n\}$,
\begin{align*}
& P(U_1^{(p)} \leq u_1,\ldots, U_k^{(p)} \leq u_k)\\
 ={}& \int_0^{u_1} \ldots \int_0^{u_k}
\frac{p^k \Gamma(n/p)}{\Gamma^k(1/p) \Gamma((n-k)/p)} (1 - \sum\nolimits_{i=1}^k v_i^p)^{(n-k)/p - 1} \, \mathrm{d} v_k \cdots \mathrm{d} v_1 \\
={}& \frac{n-k}{n} \frac{\Gamma(1+n/p)}{\Gamma^k(1+ 1/p) \Gamma(1+(n-k)/p)}
\int_0^{u_1} \ldots \int_0^{u_k}
 (1 - \sum\nolimits_{i=1}^k v_i^p)^{(n-k)/p - 1} \, \mathrm{d} v_k \cdots \mathrm{d} v_1 \\
 \to{}& \frac{n-k}{n} \prod_{i=1}^k u_i = P(U_1^{\infty} \leq u_1,\ldots, U_k^{\infty} \leq u_k)
\end{align*}
as $p \to \infty$. Thus, we have that $(U_1^{(p)},\ldots, U_n^{(p)}) \to_d (U_1^{\infty},\ldots, U_n^{\infty})$. Furthermore, using again the identity $\Gamma(x+1) = x \cdot \Gamma(x)$ and the formula for the density of the radius distribution given in Theorem \ref{thm:lps-multi} we get for $0 < r < 1$
\begin{align}
f_{R^{(p)}}(r) & = \frac{\Gamma(1/p)}{\Gamma(n/p) \Gamma(1 - \frac{n-1}p)}  \cdot r^{n-1} \cdot (1-r^p)^{-\frac{n-1}{p}} \nonumber \\
& = n \cdot \frac{\Gamma(1+1/p)}{\Gamma(1+n/p) \Gamma(1 - \frac{n-1}p)}  \cdot r^{n-1} \cdot (1-r^p)^{-\frac{n-1}{p}} \nonumber \\
& \to n  \cdot r^{n-1} \nonumber
\end{align}
as $p \to \infty$ uniformly on every set of the form $[0,r_0]$ with $r_0 \in (0,1)$. Hence we get
\begin{equation}
P(R^{(p)} \leq r) =  \int_0^{r} f_{R^{(p)}}(x) \dd x \to r^n 
\label{eq:r-lim}
\end{equation}
as $p \to \infty$. Obviously, the right-hand side of \eqref{eq:r-lim} is equal to the cumulative distribution function of $R^{\infty} = \max\{U_1^\bot,\ldots, U_n^\bot\}$.

\bibliographystyle{artbibst}
\bibliography{lpsym}

\end{document}